\documentclass[12pt]{amsart}
\usepackage{amssymb,latexsym,graphicx}
\usepackage{enumerate}

\makeatletter
\@namedef{subjclassname@2010}{%
  \textup{2010} Mathematics Subject Classification}
\makeatother

\newtheorem{thm}{Theorem}
\newtheorem{cor}[thm]{Corollary}
\newtheorem{prop}[thm]{Proposition}
\newtheorem{lem}[thm]{Lemma}

\newtheorem{question}{Question}

\theoremstyle{definition}

\newcommand{\A}{\ensuremath{{\mathbb{A}}}}
\newcommand{\C}{\ensuremath{{\mathbb{C}}}}
\newcommand{\Z}{\ensuremath{{\mathbb{Z}}}}
\renewcommand{\P}{\ensuremath{{\mathbb{P}}}}
\newcommand{\Q}{\ensuremath{{\mathbb{Q}}}}

\newcommand{\F}{\ensuremath{{\mathbb{F}}}}

\begin{document}


\baselineskip=17pt



\title[Zeta Functions of Polynomials on $\overline{\F}_p$]{Transcendence of the Artin-Mazur Zeta Function for Polynomial Maps of $\A^1(\overline{\F}_p)$}

\author[A. Bridy]{Andrew Bridy}
\address{Andrew Bridy\\Department of Mathematics\\ University of Wisconsin-Madison\\
Madison, WI 53706, USA}
\email{bridy@math.wisc.edu}

\date{}

\begin{abstract}
We study the rationality of the Artin-Mazur zeta function of a dynamical system defined by a polynomial self-map of $\A^1(\overline{\F}_p)$, where $\overline{\F}_p$ is the algebraic closure of the finite field $\mathbb{F}_p$.  The zeta functions of the maps $x\mapsto x^m$ for $p\nmid m$ and $x\mapsto x^{p^m}+ax$ for $p$ odd, $a\in\mathbb{F}_{p^m}^\times$, are shown to be transcendental.
\end{abstract}



\maketitle

\section{Definitions and Preliminaries}
In the study of dynamical systems the Artin-Mazur zeta function is the generating function for counting periodic points.  For any set $X$ and map $f:X\to X$ it is a formal power series defined by
\begin{equation}
\zeta_f(X;t)=\exp\left( \sum_{n=1}^\infty \#(\text{Fix}(f^n))\frac{t^n}{n} \right).
\end{equation} 
We use the convention that $f^n$ means $f$ composed with itself $n$ times, and that $\text{Fix}(f^n)$ denotes the set of fixed points of $f^n$.  For $\zeta_f(X;t)$ to make sense as a formal power series we assume that $\#(\text{Fix}(f^n))<\infty$ for all $n$.  The zeta function is also represented by the product formula
\begin{equation*}
\zeta_f(X;t)=\prod_{x\in \text{Per}(f,X)}(1-t^{p(x)})^{-1}
\end{equation*}
where Per$(f,X)$ is the set of periodic points of $f$ in $X$ and $p(x)$ is the least positive $n$ such that $f^n(x)=x$.  This function was introduced by Artin and Mazur in the case where $X$ is a manifold and $f:X\to X$ is a diffeomorphism~\cite{AM}.  In this context $\zeta_f(X;t)$ is proved to be a rational function for certain classes of diffeomorphisms (e.g.~\cite{G,M}).  This shows that in these cases the growth of $\#(\text{Fix}(f^n))$ is determined by the finitely many zeros and poles of $\zeta_f$.  From this point onward we make the definition\begin{equation*}
a_n=\#(\text{Fix}(f^n)) 
\end{equation*}
for economy of notation.
\newline

We are interested in the rationality of the zeta function in an algebraic context, motivated by the following example.\newline

\noindent\textbf{Example:}  Let $X$ be a variety over $\F_p$ and let $f:X\to X$ be the Frobenius map, i.e. the $p$-th power map on coordinates.  Fix$(f^n)$ is exactly the set of $\F_{p^n}$-valued points of $X$.  Therefore $\zeta_f(X;t)$ is the Hasse-Weil zeta function of $X$, and is rational by Dwork's Theorem~\cite{Dwork}.\newline

We study a simple, yet interesting case: fix a prime $p$ and let $X=\A^1_{\F_p}$, the affine line over $\F_p$.  Let $f\in\overline{\F}_p[x]$, let $d=\deg f$, and assume that $d\geq 2$.  Consider the dynamical system defined by $f$ as a self-map of $\A^1(\overline{\F}_p)$.  The points in Fix($f^n$) are the roots in $\overline{\F}_p$ of the degree $d^n$ polynomial $f^n(x)-x$ counted \emph{without multiplicity}, so $a_n\leq d^n$.  If we consider $\zeta_f(t)$ as a function of a complex variable $t$, it converges to a holomorphic function on $\C$ in a disc around the origin of radius $d^{-1}$ (at least - it is not clear that $d^{-1}$ is the largest radius of convergence).  Our motivating question is:

\begin{question}
For which $f\in\overline{\F}_p[x]$ is $\zeta_f(\overline{\F}_p;t)$ a rational function?
\end{question}

If we count periodic points with multiplicity, then $a_n=d^n$ for all $n$ and Question 1 becomes completely trivial by the calculation

\begin{equation}\label{rationalZeta}
\zeta_f(\overline{\F}_p;t)=\exp\left( \sum_{n=1}^\infty\frac{d^nt^n}{n}\right)=\exp(-\log(1-dt))=\frac{1}{1-dt},
\end{equation}
so we count each periodic point only once.  A partial answer to our question is given by the following two theorems, which show that for some simple choices of $f$, $\zeta_f$ is not only irrational, but also not algebraic over $\Q(t)$.

\begin{thm}  
If $f\in\overline{\F}_p[x^p]$, then $\zeta_f(\overline{\F}_p,t)\in\Q(t)$.  In particular, if $p\mid m$, then $\zeta_{x^m}(\overline{\F}_p;t)\in\Q(t)$.  If $p\nmid m$, then $\zeta_{x^m}(\overline{\F}_p;t)$ is transcendental over $\Q(t)$.
\end{thm}

\begin{thm}
If $a\in\F_{p^m}^\times$, $p$ odd and $m$ any positive integer, then $\zeta_{x^{p^m}+ax}(\overline{\F}_p;t)$ is transcendental over $\Q(t)$.
\end{thm}

Our strategy of proof depends heavily on the following two theorems.  Their proofs, as well as a good introduction to the theory of finite automata and automatic sequences, can be found in~\cite{AS}.

\begin{thm}[Christol]
The formal power series $\sum_{n=0}^\infty b_n t^n$ in the ring $\F_p[[t]]$ is algebraic over $\F_p(t)$ iff its coefficient sequence $\{b_n\}$ is $p$-automatic.
\end{thm}

\begin{thm}[Cobham]
For $p$, $q$ multiplicatively independent positive integers (i.e. $\log p/\log q\notin\Q$), the sequence $\{b_n\}$ is both $p$-automatic and $q$-automatic iff it is eventually periodic.
\end{thm}

The following is an easy corollary to Christol's theorem which we will use repeatedly~\cite[Theorem 12.6.1]{AS}.
\begin{cor}\label{ChristolCor}
If $\sum_{n=0}^\infty b_n t^n\in\Z[[t]]$ is algebraic over $\Q(t)$, then the reduction of $\{b_n\}$ mod $p$ is $p$-automatic for every prime $p$.
\end{cor}
We note that Corollary \ref{ChristolCor} will be applied to the logarithmic derivative $\zeta_f'/\zeta_f=\sum_{n=1}^\infty a_n t^{n-1}$, rather than to $\zeta_f$.\newline

Throughout this paper we use $v_p$ to mean the usual $p$-adic valuation, that is, $v_p(a/b)=\text{ord}_p(b)-\text{ord}_p(a)$.  We use $(n)_p$ as in \cite{AS} to signify the base-$p$ representation of the integer $n$, and we denote the multiplicative order of $a$ mod $n$ by $o(a,n)$, assuming that $a$ and $n$ are coprime integers.\newline

\section{Proof of Theorem 1}
\begin{proof}
Let $f(x)\in\overline{\F}_p[x^p]$, so that $f'(x)=0$ identically.  Then $f^n(x)-x$ has derivative $(f^n(x)-x)'=-1$, so it has distinct roots over $\overline{\F}_p$. Therefore $a_n=(\deg f)^n$ and $\zeta_f(\overline{\F}_p,t)$ is rational as in equation (\ref{rationalZeta}).\newline

Now suppose $f(x)=x^m$ where $p\nmid m$.  Assume by way of contradiction that $\zeta_f$ is algebraic over $\Q(t)$.  The derivative $\zeta_f'=d\zeta_f/dt$ is algebraic, which can be shown by writing the polynomial equation that $\zeta_f$ satisfies and applying implicit differentiation.  Hence $\zeta_f'/\zeta_f$ is algebraic.  We have
\begin{equation*}
\zeta_f'/\zeta_f=(\log\zeta_f)'=\sum_{n=1}^\infty a_n t^{n-1}
\end{equation*}
so in particular, $\zeta_f'/\zeta_f\in\Z[[t]]$.  By Corollary \ref{ChristolCor}, for every prime $q$ the reduced sequence $\{a_n\}$ mod $q$ is $q$-automatic.\newline

First we count the roots of $f^n(x)-x=x^{m^n}-x=x(x^{m^n-1}-1)$ in $\overline{\F}_p$.  There is one root at zero, and we write $m^n-1=p^ab$, where $p\nmid b$, so 
\begin{equation*}
x^{m^n-1}-1=x^{p^ab}-1=(x^b-1)^{p^a}.
\end{equation*}
The polynomial $x^b-1$ has derivative $bx^{b-1}$, and $(x^b-1,bx^{b-1})=1$, so $x^b-1$ has exactly $b$ roots in $\overline{\F}_p$, as does $x^{m^n}-1$.  Therefore
\begin{equation}\label{a_nCalculation}
a_n=1+\frac{m^n-1}{p^{v_p(m^n-1)}}.
\end{equation}
Now we need to reduce mod some carefully chosen prime $q$.  There are two cases to consider, depending on whether $p=2$.\newline

\noindent\underline{Case 1:} If $p=2$, let $q$ be a prime dividing $m$, $q\neq 2$.  There is such a prime because $m>1$ and $2\nmid m$.  Let $r = 2^{-1}$ in $\mathbb{F}_q$.  Reducing mod $q$,
\begin{equation}\label{a_nCase1}
a_n= 1+ \frac{m^n-1}{2^{v_2(m^n-1)}}\equiv 1-r^{v_2(m^n-1)}\pmod{q}.
\end{equation}
The subsequence $\{a_{2n}\}$ reduced mod $q$ is $q$-automatic because subsequences of automatic sequences indexed by arithmetic progressions are automatic~\cite[Theorem 6.8.1]{AS}.  We define the sequence $\{b_n\}$ as
\begin{equation*}
b_n=-(a_{2n}-1).
\end{equation*}
The sequence $\{b_n\}$ is $q$-automatic, because subtracting $1$ and multiplying by $-1$ simply permute the elements of $\mathbb{F}_q$.  We have $b_n=r^{v_2(m^{2n}-1)}$ by (\ref{a_nCase1}).  To proceed, we need the following proposition.\newline

\begin{prop}\label{v_pFormula}
\begin{enumerate}[i.]
\item For any $n,m\in\mathbb{N}$, $m$ odd, 
$$
v_2(m^{2n}-1)=v_2(n)+v_2(m^2-1).
$$
\item If $p$ is an odd prime and $n,m\in\mathbb{N}$, $p\nmid m$, then 
$$
v_p(m^{(p-1)n}-1)=v_p(n)+v_p(m^{p-1}-1).
$$
\end{enumerate}
\end{prop}

\begin{proof}
The proof is an elementary consequence of the structure of the unit group $(\Z/p^n\Z)^\times$, see for example \cite{Le}, and is omitted.
\end{proof}

By Proposition ~\ref{v_pFormula},
\begin{equation}\label{b_n}
b_n=r^{v_2(n)+v_2(m^2-1)}.
\end{equation}  Let $d=o(r,q)$, the multiplicative order of $r$ in $\mathbb{F}_q$, and note that $d>1$ because $r\neq 1$.  We see that $b_n$ is a function of $v_2(n)$ reduced mod $d$, and $v_2(n)$ is simply the number of leading zeros of $(n)_2$ (if we read the least significant digit first).

\begin{lem}\label{automatic}
If $\beta_n$ is a function of the equivalence class mod $d$ of $v_p(n)$, then the sequence $\{\beta_n\}$ is $p$-automatic.
\end{lem}

\begin{proof}
We can build a finite automaton (with output) whose output depends on the equivalence class mod $d$ of the number of initial zeros of a string, as in Figure 1 for $d=4$.  There are $d$ states arranged in a circle (the $q_i$ in the figure), reading a zero moves from one of these states to the next, and reading any other symbol moves to a final state (the $r_i$) marked with the corresponding output.  Therefore $\beta_n$ is $p$-automatic.
\end{proof}

\begin{figure}[htb]
\includegraphics[scale=0.65]{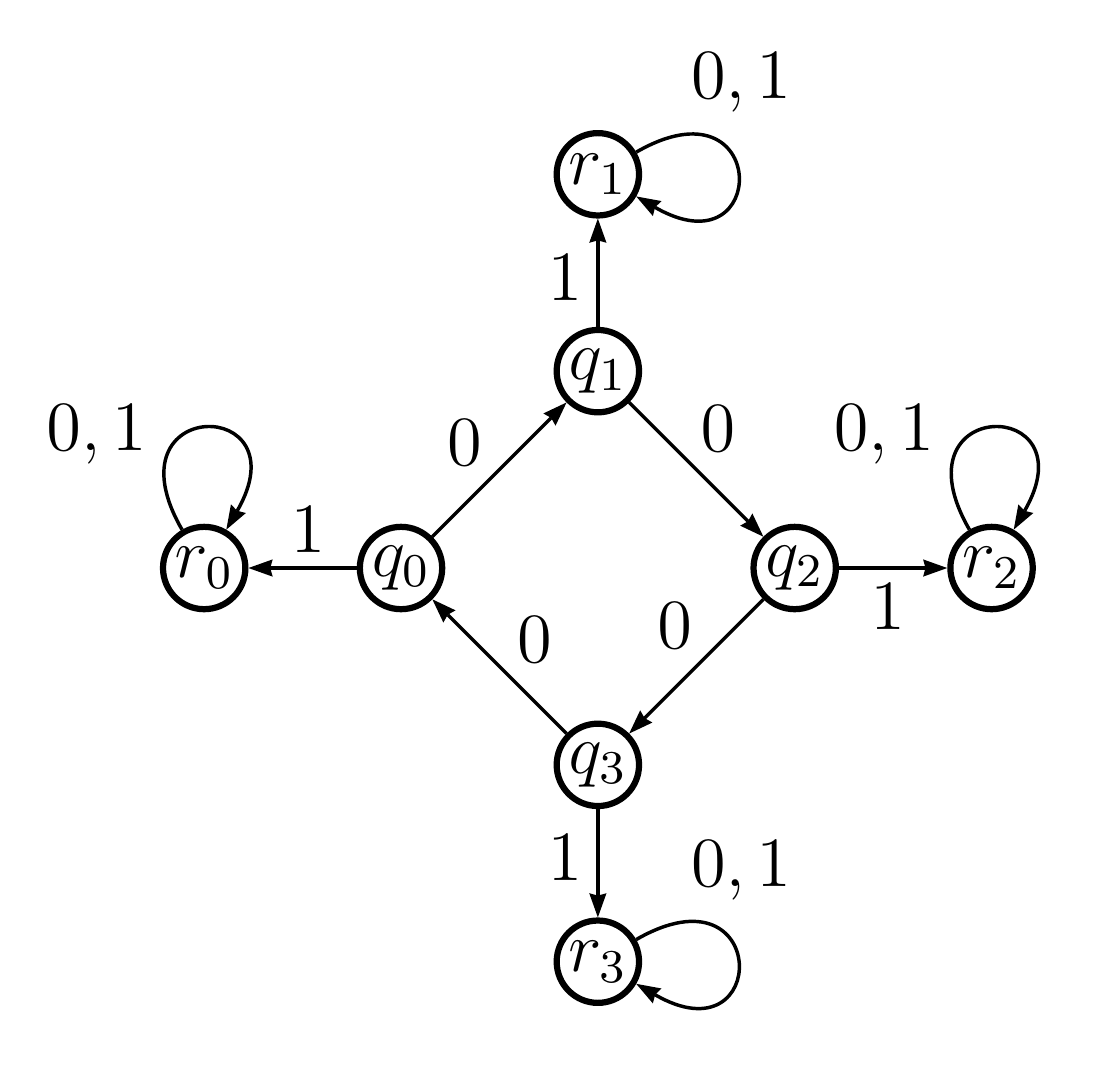}
\caption{State $q_0$ is initial. States $q_i$ and $r_i$ are reached after processing $i\bmod{4}$ leading zeroes.}
\end{figure}

By Lemma \ref{automatic}, $\{b_n\}$ is 2-automatic.  It is also $q$-automatic, so by Cobham's theorem $\{b_n\}$ is eventually periodic of period $k$.  For some large $n$, we have $b_{nk}=b_{nk+k}=b_{nk+2k}=\dots=b_{(n+a)k}$ for any positive integer $a$.  This means that $b_{Nk}=b_{nk}$ for all $N>n$.  By equation (\ref{b_n}),
\begin{equation*}
r^{v_2(Nk)+v_2(m^2-1)}=r^{v_2(nk)+v_2(m^2-1)}
\end{equation*} 
which means $v_2(Nk)\equiv v_2(nk)\pmod{d}$ and so $v_2(N)\equiv v_2(n)\pmod{d}$ for all $N>n$.  This is a contradiction, as $d>1$.\newline

\noindent\underline{Case 2:} If $p>2$, we pick some prime $q>m^{p-1}$ such that $q\not\equiv 1\pmod{p}$ (for example we can choose $q\equiv 2\pmod{p}$ by Dirichlet's theorem on primes in arithmetic progressions).  Clearly $q\nmid m$, so $m^{q-1}\equiv 1\pmod{q}$.  Let $r= p^{-1}$ in $\mathbb{F}_q$.  The sequence $\{a_n\}$ is as in equation (\ref{a_nCalculation}).  We take the subsequence $a_{(p-1)((q-1)n+1)}$ and reduce it mod $q$.  This subsequence is $q$-automatic.  We compute
\begin{align*}
a_{(p-1)((q-1)n+1)} & = 1 + \frac{m^{(p-1)((q-1)n+1)}-1}{p^{v_p(m^{(p-1)((q-1)n+1)}-1)}}  = 1 + \frac{(m^{q-1})^{(p-1)n}m^{p-1}-1}{p^{v_p(m^{(p-1)((q-1)n+1)}-1)}}\\
& \equiv 1 + (m^{p-1}-1)r^{v_p(m^{(p-1)((q-1)n+1)}-1)}\pmod{q}.
\end{align*}
As $m^{p-1}-1<q$ we can invert $m^{p-1}-1$ mod $q$.  If we subtract 1 and multiply by $(m^{p-1}-1)^{-1}$ as in Case 1, we get $$b_n=r^{v_p(m^{(p-1)((q-1)n+1)}-1)}$$ which is $q$-automatic.\newline

By Proposition \ref{v_pFormula}, $b_n=r^{v_p((q-1)n+1)+v_p(m^{p-1}-1)}$.  Let $d=o(r,q)$, noting that $d>1$.  Let 
\begin{equation*}
Y=\{n\in\mathbb{N}: v_p((q-1)n+1)\equiv 0\pmod{d}\}.
\end{equation*}
$Y$ is the fiber of $\{b_n\}$ over $r^{v_p(m^{p-1}-1)}$ and is therefore a $q$-automatic set (i.e. its characteristic sequence is $q$-automatic).  We argue that $Y$ is $p$-automatic.\newline

Consider a finite-state transducer $T$ on strings over $\{0,\dots,p-1\}$ such that $T((n)_p)=((q-1)n+1)_p$.  On strings with no leading zeros, $T$ is one-to-one.  Let $L$ be the set of base-$p$ strings $(n)_p$ such that $n\in Y$.  Then 
\begin{equation*}
T(L)=\{(n)_p: n\equiv 1\pmod{q-1}\hspace{.2in}\text{and}\hspace{.2in} v_p(n)\equiv 0\pmod{d}\}.
\end{equation*}
$T(L)$ is a regular language, as both of its defining conditions can be recognized by a finite automaton (for the second condition, this follows from Lemma \ref{automatic}).  Therefore $T^{-1}(T(L))=L$ is regular, that is, the characteristic sequence of $Y$ is $p$-automatic.  We use Cobham's theorem again to conclude that the characteristic sequence of $Y$ is eventually periodic.\newline

Let $\{y_n\}$ be the characteristic sequence of $Y$
$$y_n= \left\{
     \begin{array}{lr}
       1 & : n \in Y\\
       0 & : n \notin Y
     \end{array}
   \right.$$
and let $k$ be its (eventual) period.  Write $k$ as $k=Mp^N$, where $p\nmid M$ (it is possible that $N=0$).  As $q\not\equiv 1\pmod{p}$, $q-1$ is invertible mod $p$-powers, so we can solve the following equation for $n$.
\begin{equation}\label{Case2Eq1}
(q-1)n\equiv -1 + p^{dN}\pmod{p^{dN+2}}
\end{equation}
Any $n$ that solves this equation satisfies $v_p((q-1)n+1)=dN$ and so $y_n=1$.  Choose a large enough solution $n$ so that $\{y_n\}$ is periodic at $n$.  We can solve the following equation for $a$, and choose such an $a$ to be positive.
\begin{equation}\label{Case2Eq2}
(q-1)aM\equiv p^{(d-1)N}(p-1)\pmod{p^{dN+2}}
\end{equation}
Multiplying (\ref{Case2Eq2}) by $p^N$ gives
\begin{equation}\label{Case2Eq3}
(q-1)ak \equiv p^{dN+1}-p^{dN} \pmod{p^{dN+2}}.
\end{equation}
Adding (\ref{Case2Eq1}) and (\ref{Case2Eq3}) gives
\begin{equation*}
(q-1)(n+ak)\equiv -1 + p^{dN+1}\pmod{p^{dN+2}}
\end{equation*}
from which we conclude $v_p((q-1)(n+ak)+1)=dN+1$.  So $y_{n+ak}=0$.  But $y_n=y_{n+ak}$ by periodicity, which is a contradiction.
\end{proof}

\section{Proof of Theorem 2}

\begin{proof}

Let $f(x)=x^{p^m}+ax$ for $a\in\mathbb{F}_{p^m}^\times$, $p$ odd.  First we compute $f^n(x)$.

\begin{prop}
$f^n(x)=\sum_{k=0}^n {n\choose k} x^{p^{km}}a^{n-k}$
\end{prop}

\begin{proof}
Let $\phi(x)=x^{p^m}$ and $a(x)=ax$, so $f=\phi+a$.  Both $\phi$ and $a$ are additive polynomials (they distribute over addition) and they commute, so the proof is simply the binomial theorem applied to $(\phi+a)^n$. 
\end{proof}

Assume that $\zeta_f$ is algebraic.  By Corollary ~\ref{ChristolCor}, the sequence $\{a_n\}$ reduced mod $q$ is $q$-automatic for every prime $q$, as is the subsequence $\{a_{(p^m-1)n}\}$ by previous remarks.  Now we need to compute $a_n$ when $p^m-1$ divides $n$.

\begin{prop}
If $p^m-1$ divides $n$, then $a_{n}=p^{(n-p^{v_p(n)})m}$.
\end{prop}

\begin{proof}
The coefficient on $x$ in $f^{n}(x)$ is a power of $a^{p^m-1}=1$.  Let $l$ be the smallest positive integer such that ${n\choose l}\not\equiv 0\pmod{p}$.  Then 
\begin{equation*}
f^{n}(x)-x=\sum_{k=l}^{n} {n\choose k} x^{p^{km}}a^{n-k}
=\left(\sum_{k=l}^{n}{n\choose k} x^{p^{(k-l)m}}(a^{n-k})^{p^{-l}}\right)^{p^l},
\end{equation*}
where raising to the $p^{-l}$ power means applying the inverse of the Frobenius automorphism $l$ times.  Let $g(x)=\sum_{k=l}^{n}{n\choose k} x^{p^{(k-l)m}}(a^{n-k})^{p^{-l}}$.  The derivative $g'(x)=(a^{n-l})^{p^{-l}}$ is nonzero, so $g(x)$ has $p^{(n-l)m}$ distinct roots over $\overline{\F}_p$, as does $f^n(x)-x$.  So $a_n=p^{(n-l)m}$.\newline

Kummer's classic theorem on binomial coefficients mod $p$ says that $v_p({n\choose l})$ equals the number of borrows involved in subtracting $l$ from $n$ in base $p$~\cite{K}.  It is clear that the smallest integer $l$ that results in no borrows in this subtraction is $l=p^{v_p(n)}$, and we are done.\end{proof}

Let $q>p$ be a prime to be determined and let $r=p^{-1}$ in $\mathbb{F}_q$.  The sequence given by $b_n=r^{(p^m-1)nm}$ is eventually periodic and so is $q$-automatic.  Let $c_n=a_{(p^m-1)n}b_n$.  By ~\cite[Corollary 5.4.5]{AS} the product of $q$-automatic sequences over $\mathbb{F}_q$ is $q$-automatic, so $c_n$ is $q$-automatic.  So
\begin{align*}
c_n & =a_{(p^m-1)n}b_n=p^{((p^m-1)n-p^{v_p((p^m-1)n)})m}r^{(p^m-1)nm}\\
 & =(p^{-1})^{p^{(v_p(p^m-1)+v_p(n)})m}=(r^m)^{p^{v_p(n)}}.
\end{align*}
Choose $q>p^{mp}$ such that $q\equiv 2\pmod{p^m}$. Note that $o(r^m,q)$ divides $q-1$, so $o(r^m,q)\not\equiv 0\pmod{p}$ and $p$ is invertible mod $o(r^m,q)$.  The value of $c_n$ depends only on $p^{v_p(n)}$ reduced mod $o(r^m,q)$, which in turn is a function of $v_p(n)$ mod $o(p,o(r^m,q))$, so $c_n$ is $p$-automatic by Lemma ~\ref{automatic}.\newline

By Cobham's Theorem $c_n$ is eventually periodic, so the set 
\begin{align*}
Y & =\{n\in\mathbb{N}:c_n=r^m\}\\
  & =\{n\in\mathbb{N}:p^{v_p(n)}\equiv 1\pmod{o(r^m,q)}\}\\
  & =\{n\in\mathbb{N}:v_p(n)\equiv 0\pmod{o(p,o(r^m,q))}\}
\end{align*}has an eventually periodic characteristic sequence $\{y_n\}$.  Essentially the same argument as in Theorem 1, Case 2 shows this is a contradiction when $o(p,o(r^m,q))>1$.  We sketch the argument for completeness.\newline

As we chose $q>p^{mp}$, $o(r^m,q)=o(p^m,q)>p$, and $o(p,o(r^m,q))>1$.  Let $d=o(p,o(r^m,q))$, and let $k=Mp^N$ be the eventual period of $Y$, where $p\nmid M$.  We can solve
\begin{equation}\label{Thm2Eq1}
n\equiv p^{dN}\pmod{p^{dN+2}}
\end{equation}
\begin{equation}\label{Thm2Eq2}
aM\equiv p^{(d-1)N}(p-1)\pmod{p^{dN+2}}
\end{equation}
for large $n$ and positive $a$, so $y_n=1$.  Adding (\ref{Thm2Eq1}) and $p^N$ times (\ref{Thm2Eq2})  gives
\begin{equation*}
n+ak\equiv p^{dN+1}\pmod{p^{dN+2}}
\end{equation*}
from which we conclude $v_p(n+ak)=dN+1$, so $y_{n+ak}=0$, contradicting periodicity of $\{y_n\}$.  This contradiction shows that $\zeta_f$ is transcendental.
\end{proof}

\section{Concluding Remarks}
The polynomial maps in Theorems 1 and 2 are homomorphisms of the multiplicative and additive groups of $\overline{\F}_p$, respectively.  It should be possible to prove similar theorems for other maps associated to homomorphisms, e.g. Chebyshev polynomials, general additive polynomials, and Latt\`es maps on $\P^1(\overline{\F}_p)$.  See ~\cite{Silverman} for a discussion of special properties of these maps.\newline

It is more difficult to study the rationality or transcendence of $\zeta_f$ when the map $f$ has no obvious structure.  For example, there is a standard heuristic that the map $f(x)=x^2+1$ behaves like a random mapping on a finite field of odd order (see ~\cite{Bach}, ~\cite{Pollard}, ~\cite{Silverman2} and many others).  We conclude with the following tantalizing question without hazarding a guess as to the answer.

\begin{question}
For $p$ odd and $f=x^2+1$, is $\zeta_f(\overline{\F}_p,t)$ in $\Q(t)$?
\end{question}

\subsection*{Acknowledgements}
This research was partly supported by NSF grant no. CCF-0635355.  The author wishes to thank Eric Bach for many helpful suggestions and comments, Jeff Shallit for useful clarifications, and an anonymous referee for helpful remarks on style and presentation.

\end{document}